\documentclass[11pt]{amsart}

\usepackage{amsmath,amssymb,latexsym,soul,cite,mathrsfs}

\usepackage{color,enumitem,graphicx}
\usepackage[colorlinks=true,urlcolor=blue,
citecolor=red,linkcolor=blue,linktocpage,pdfpagelabels,
bookmarksnumbered,bookmarksopen]{hyperref}
\usepackage[english]{babel}

\usepackage[left=2.9cm,right=2.9cm,top=2.8cm,bottom=2.8cm]{geometry}
\usepackage[hyperpageref]{backref}

\usepackage[colorinlistoftodos]{todonotes}
\makeatletter
\providecommand\@dotsep{5}
\def\listtodoname{List of Todos}
\def\listoftodos{\@starttoc{tdo}\listtodoname}
\makeatother

\numberwithin{equation}{section}

\newtheorem{theorem}{Theorem}[section]
\newtheorem{proposition}[theorem]{Proposition}
\newtheorem{lemma}[theorem]{Lemma}

\newtheorem{remark}{Remark}

\newcommand\R{\mathbb R}

\begin{document}

\title[fractional Hamiltonian systems with positive semi-definite matrix]
{Existence and concentration of solution for a fractional Hamiltonian systems with positive semi-definite matrix}

\author{C\'esar Torres}
\author{Ziheng Zhang}
\author{Amado Mendez}

\address[C\'esar Torres]{\newline\indent Departamento de Matem\'aticas
\newline\indent 
Universidad Nacional de Trujillo,
\newline\indent
Av. Juan Pablo II s/n. Trujillo-Per\'u}
\email{\href{mailto:ctl\_576@yahoo.es}{ctl\_576@yahoo.es}}

\address[Ziheng Zhang]
{\newline\indent Department of Mathematics, 
\newline\indent 
Tianjin Polytechnic University, 
\newline\indent
Tianjin 300387, China.}
\email{\href{mailto:zhzh@mail.bnu.edu.cn}{zhzh@mail.bnu.edu.cn}}

\address[Amado Mendez]{\newline\indent Departamento de Matem\'aticas
\newline\indent 
Universidad Nacional de Trujillo,
\newline\indent
Av. Juan Pablo II s/n. Trujillo-Per\'u}
\email{\href{mailto:gamc55@hotmail.com}{gamc55@hotmail.com}}


%

\pretolerance10000


\begin{abstract}
\noindent We study the existence of solutions for the following fractional Hamiltonian systems
$$
\left\{
  \begin{array}{ll}
   - _tD^{\alpha}_{\infty}(_{-\infty}D^{\alpha}_{t}u(t))-\lambda L(t)u(t)+\nabla W(t,u(t))=0,\\[0.1cm]
    u\in H^{\alpha}(\mathbb{R},\mathbb{R}^n),
  \end{array}
\right.
  \eqno(\mbox{FHS})_\lambda
$$
where $\alpha\in (1/2,1)$, $t\in \mathbb{R}$, $u\in \mathbb{R}^n$, $\lambda>0$ is a parameter, $L\in C(\mathbb{R},\mathbb{R}^{n^2})$ is a symmetric matrix
for all $t\in \mathbb{R}$, $W\in C^1(\mathbb{R} \times \mathbb{R}^n,\mathbb{R})$. Assuming that
$L(t)$ is a positive semi-definite symmetric matrix for all $t\in \mathbb{R}$, that is, $L(t)\equiv 0$ is allowed to occur in some finite interval $T$ of $\mathbb{R}$,
$W(t,u)$ satisfies some superquadratic conditions weaker than Ambrosetti-Rabinowitz condition, we show that (FHS)$_\lambda$ has a solution which vanishes on
$\mathbb{R}\setminus T$ as $\lambda \to \infty$, and converges to some $\tilde{u}\in H^{\alpha}(\R, \R^n)$. Here, $\tilde{u}\in E_{0}^{\alpha}$ is a solution
of the Dirichlet BVP for fractional systems on the finite interval $T$. Our results are new and improve recent results in the literature even in the case $\alpha =1$.
\end{abstract}

\subjclass[2010]{Primary 34C37; Secondary 35A15, 35B38.} 
\keywords{Fractional Hamiltonian systems, Fractional Sobolev space, Critical point theory, Concentration phenomena.}

\maketitle

\section{Introduction}

Fractional Hamiltonian systems are a significant area of nonlinear analysis, since they appear in many phenomena studied in several fields of applied science,
such as engineering, physics, chemistry, astronomy and control theory. On the other hand, the theory of fractional calculus is a part that intensively developing
during the last decades; see  \cite{ATMS04, ErvinR06, Hilfer00, KST06, MR93, Pod99} and the references therein. The existence of homoclinic solutions for Hamiltonian systems and their importance in the study of behavior of dynamical
systems can be recognized from Poincar\'{e} \cite{Poincare}. Since then, the investigation of the existence and multiplicity of homoclinic solutions became one of the main 
important problems of research in dynamical systems. Critical point theorem was first used by Rabinowitz \cite{Rab86} to obtain the existence of periodic solutions for first
order Hamiltonian systems, while the first multiplicity result is due to Ambrosetti and Zelati \cite{AmbroZelati93}. Therefore, a large number of mathematicians used critical point
theory and variational methods to prove the existence of homoclinic solutions for Hamiltonian systems; see for instance  \cite{Co91,GWC,Ding95,Izydorek05,Izydorek07,
Omana92,Rab90,Rab91} and the reference therein. 

The critical point theory has become an effective tool in investigating the existence and multiplicity of solutions for fractional differential equations by 
constructing fractional variational structures. Especially, in \cite{FJYZ} the authors firstly dealt with a class of fractional boundary value problem via critical 
point theory. From then on, Variational methods and critical point theory are shown to be effective in determining the solutions for fractional differential 
equations with variational structure. We also mention the work by Torres \cite{Torres12}, where the author considered the following fractional Hamiltonian 
systems
$$
\left\{
  \begin{array}{ll}
    _tD^{\alpha}_{\infty}(_{-\infty}D^{\alpha}_{t}u(t))+L(t)u(t)=\nabla W(t,u(t)),\\[0.1cm]
    u\in H^{\alpha}(\mathbb{R},\mathbb{R}^n),
  \end{array}
\right.
  \eqno(\mbox{FHS})
$$
where $\alpha\in (1/2,1)$, $t\in \mathbb{R}$, $u\in \mathbb{R}^n$, $L\in C(\mathbb{R},\mathbb{R}^{n^2})$ is a symmetric and positive definite matrix
for all $t\in \mathbb{R}$, $W\in C^1(\mathbb{R}\times \mathbb{R}^n,\mathbb{R})$ and $\nabla W(t,u)$ is the gradient of $W(t,u)$ at $u$. Assuming that $L(t)$ satisfied the following coercivity condition
\begin{itemize}
\item[(L)] there exists an $l\in C(\mathbb{R},(0,\infty))$ with $l(t)\rightarrow \infty$ as $|t|\rightarrow \infty$ such that
\begin{equation}\label{eqn:L coercive}
(L(t)u,u)\geq l(t)|u|^2 \quad \mbox{for all}\,\, t\in \mathbb{R} \,\, \mbox{and} \,\, u\in \mathbb{R}^n.
\end{equation}
\end{itemize}
and that $W(t,u)$ satisfies the Ambrosetti-Rabinowitz condition
\begin{itemize}
\item[(FHS$_1$)]$W\in C^1(\mathbb{R} \times \mathbb{R}^n,\mathbb{R})$ and there is a constant $\theta>2$ such that
$$
0<\theta W(t,u)\leq (\nabla W(t,u),u)\quad \mbox{for all}\,\, t\in \mathbb{R} \,\,\mbox{and}\,\, u\in \mathbb{R}^n\backslash\{0\},
$$
\end{itemize}
and other suitable conditions, the author showed that (FHS) possesses at least one nontrivial solution via Mountain
Pass Theorem. Note that (FHS)$_1$, implies that $W(t,u)$ is of superquadratic growth as $|u|\rightarrow \infty$.
Since then, many researchers dealt with (FHS) for the cases that $W(t,u)$ is superquadratic or subquadratic at infinity; see for instance \cite{MendezTorres15,XuReganZhang15,ZhangYuan}. In addition, some perturbed fractional Hamiltonian systems are discussed in \cite{Torres14,XuReganZhang15}.

In \cite{ZhangYuan14} the authors focused on weakening the coercivity condition $(L)$, more precisely they assumed that $L(t)$ is bounded in the following sense:
\begin{itemize}
\item[(L)$'$] $L\in C(\mathbb{R},\mathbb{R}^{n^2})$ is a symmetric and positive definite matrix for all $t\in \mathbb{R}$ and there are constants $0<\tau_1<\tau_2<\infty$ such that
$$
\tau_1|u|^2\leq (L(t)u,u)\leq \tau_2|u|^2\quad \mbox{for all}\,\, (t,u)\in \mathbb{R} \times \mathbb{R}^n,
$$
\end{itemize}
By supposed that $W(t,u)$ is subquadratic as $|u|\rightarrow +\infty$, the authors also showed that (FHS) possessed infinitely many solutions, which has been generalized in
\cite{NyaZhou2017,ZhouZhang2017}.

In the present paper we deal with the following fractional Hamiltonian systems
$$
\left\{
  \begin{array}{ll}
   - _tD^{\alpha}_{\infty}(_{-\infty}D^{\alpha}_{t}u(t))-\lambda L(t)u(t)+\nabla W(t,u(t))=0,\\[0.1cm]
    u\in H^{\alpha}(\mathbb{R},\mathbb{R}^n),
  \end{array}
\right.
  \eqno(\mbox{FHS})_{\lambda}
$$
where $\alpha\in (1/2,1)$, $t\in \mathbb{R}$, $u\in \mathbb{R}^n$, $\lambda>0$ is a parameter, $W\in C^1(\mathbb{R} \times \mathbb{R}^n,\mathbb{R})$ and $L$ satisfies the following conditions
\begin{itemize}
\item[$(\mathcal{L})_1$]$L\in C(\mathbb{R},\mathbb{R}^{n\times n})$ is a symmetric matrix for all $t\in\mathbb{R}$; there exists a nonnegative continuous function $l:\mathbb{R} \rightarrow \mathbb{R}$ and
a constant $c>0$ such that
$$
(L(t)u,u)\geq l(t)|u|^2,
$$
and the set $\{l<c\}:=\{t\in \mathbb{R} \,|\,l(t)<c\}$ is nonempty with $meas \{l<c\}<\frac{1}{C_{\infty}^2}$, where $meas \{\cdot\}$ is the Lebesgue measure and $C_\infty$
is the best Sobolev constant for the embedding of $X^{\alpha}$ into $L^{\infty}(\mathbb{R})$;
\item[$(\mathcal{L})_2$]$J=int (l^{-1}(0))$ is a nonempty finite interval and $\overline{J}=l^{-1}(0)$;
\item[$(\mathcal{L})_3$]there exists an open interval $T\subset J$ such that $L(t)\equiv 0$ for all $t\in \overline{T}$.
\end{itemize}
In particular, if $\alpha=1$ in (FHS)$_\lambda$, then it reduces to the following well-known
second order Hamiltonian systems
$$
\ddot u- \lambda L(t) u+\nabla W(t,u)=0.\eqno(\mbox{HS})
$$
Recently a second order Hamiltonian systems like (HS) with positive semi-definite matrix was considered in \cite{JSTW}. Assuming that
$W \in C^1(\mathbb{R}\times \mathbb{R}^n, \mathbb{R})$ is an indefinite potential satisfying asymptotically quadratic condition at infinity on $u$, Sun and Wu,
with a little mistake in their embedding results, have proved the existence of two homoclinic solutions of (FHS$_\lambda$). For more related works, we refer the reader to
\cite{Co91,Ding95,Izydorek05,Izydorek07,Omana92,Rab91} and the references mentioned there.

Here we must point out, to obtain the existence or multiplicity of solutions for Hamiltonian systems, all the papers mentioned above need the assumption that
the symmetric matrix $L(t)$ is positive definite, see (L) and (L)$'$. Therefore, recently the authors in \cite{Benhassine2017,TorresZhang2017,ZhangTorres}
considered the case that  $L(t)$ is positive semi-definite satisfying $(\mathcal{L})_1$. In \cite{Benhassine2017}, the author dealt with (FHS) for the case that
$(\mathcal{L})_1$ is satisfied and $W(t,u)$ involves a combination of superquadratic and subquadratic terms and is allowed to be sign-changing. In  \cite{TorresZhang2017,ZhangTorres},
we have considered the existence of solutions of (FHS)$_\lambda$ and the concentration of its solutions when $(\mathcal{L})_1$-$(\mathcal{L})_3$ are satisfied and
$W(t,u)$ meets with some classes of superquadratic hypothesis.

Motivated by these previous results, the main purpose of this paper is
to investigate (FHS)$_\lambda$ without Ambrosetti-Rabinowitz condition (FHS$_1$). More precisely, we suppose that  $W(t,u)$ satisfy the following assumptions
\begin{enumerate}
\item[($W_1$)] $|\nabla W(t,u)| = o(|u|)$ as $|u|\to 0$ uniformly in $t\in \mathbb{R}$.
\item[$(W_2)$] $W(t,u)\geq 0$ for all $(t,u)\in \mathbb{R}\times \mathbb{R}^N$ and $H(t,u)\geq0$  for all $(t, u) \in \mathbb{R} \times \mathbb{R}^N$, where
$$
H(t,u) :=\frac{1}{2}\langle \nabla W(t,u), u \rangle - W(t,u).
$$
\item[$(W_3)$] $\frac{W(t,u)}{|u|^2} \to + \infty$ as $|u| \to +\infty$ uniformly in $t\in \mathbb{R}$.
\item[$(W_4)$] There exist $C_0, R >0$, and $\sigma >1$ such that
$$
\frac{|\nabla W(t,u)|^\sigma}{|u|^\sigma} \leq C_0H(t,u)\;\;\mbox{if}\;\;|u| \geq R.
$$
\end{enumerate}
Note that, according to \cite{GWC} the nonlinearity
$$
W(t,u) = g(t)(|u|^p + (p-2)|u|^{p-\epsilon}\sin^2(\frac{|u|^\epsilon}{\epsilon})),
$$
where $g(t)>0$ is $T$-periodic in $t$, $0<\epsilon < p-2$ and $p>2$, satisfies $(W_1)-(W_4)$, but (FHS$_1$) is not satisfied.

Now we are in the position to state our main result.
\begin{theorem}\label{Thm:MainTheorem1}
Suppose that {\rm ($\mathcal{L}$)$_1$}-{\rm ($\mathcal{L}$)$_3$}, $(W_1) - (W_4)$ are satisfied, then there exists $\Lambda _*>0$ such that for every
$\lambda>\Lambda_*$, {\rm(FHS)$_\lambda$} has at least one nontrivial solution.
\end{theorem}

On the concentration of solutions obtained above, for technical reason, we consider that there exists $0<\varrho< +\infty$, such that $T = [-\varrho,\varrho]$,
where $T$ is given by $(\mathcal{L})_3$. We have the following result.
\begin{theorem}\label{Thm:MainTheorem2}
Let $u_\lambda$ be  a solution of problem ${\rm (FHS)}_\lambda$ obtained in Theorem \ref{Thm:MainTheorem1}, then $u_\lambda \to \tilde{u}$ strongly in
$H^{\alpha}(\mathbb{R})$ as $\lambda \to \infty$, where $\tilde{u}$ is a nontrivial solution of the following boundary value problem
\begin{eqnarray}\label{eqn:BVP}
\left\{
  \begin{array}{ll}
    {_{t}}D_{\varrho}^{\alpha} ({_{-\varrho}}D_{t}^{\alpha})u  = \nabla W(t, u), & t\in (-\varrho, \varrho) \\[0.2cm]
    u(-\varrho) = u(\varrho) = 0,
  \end{array}
\right.
\end{eqnarray}
where ${_{-\varrho}}D_{t}^{\alpha}$ and $_{t}D_{\varrho}^{\alpha}$ are left and right Riemann-Liouville fractional derivatives of order $\alpha$ on $[-\varrho,\varrho]$ respectively.
\end{theorem}
\begin{remark}
{\rm
In Theorem \ref{Thm:MainTheorem1}, we give some new superquadratic conditions on $W(t,u)$ to guarantee the existence of solutions and investigate
the concentration of these solutions in \ref{Thm:MainTheorem2}. However, we must point out that the methods in \cite{Benhassine2017,TorresZhang2017,ZhangTorres}
are not be valid for our new assumptions. To overcome this difficulty we apply the Mountain Pass Theorem with Cerami condition, however, the direct application of the mountain pass theorem is not enough since the Cerami sequences might lose compactness in the whole space $\mathbb{R}$. Then it is necessary to introduce a new compactness result to recover the convergence of Cerami sequence, for more details see Lemma \ref{cerami1}.
}
\end{remark}
The remaining part of this paper is organized as follows. Some preliminary results are presented in Section 2. In Section
3, we are devoted to accomplishing the proof of Theorem \ref{Thm:MainTheorem1} and in Section 4 we present the proof of Theorem \ref{Thm:MainTheorem2}.

\section{Preliminary Results}

In this section, for the reader's convenience, firstly we introduce some basic definitions of fractional calculus.
The Liouville-Weyl fractional derivative of order $0<\alpha<1$ are defined as
\begin{equation}\label{eqn:RD}
_{-\infty}D^{\alpha}_x u(x)=\frac{d}{dx} {_{-\infty}I^{1-\alpha}_x u(x)}
\quad \mbox{and}\quad
_{x}D^{\alpha}_{\infty} u(x)=-\frac{d}{dx} {_{x}I^{1-\alpha}_{\infty} u(x)}.
\end{equation}
where $_{-\infty}I^{\alpha}_x$ and $_{x}I^{\alpha}_{\infty}$ are the left and right Liouville-Weyl fractional integrals of order $0<\alpha<1$ defined as
$$
_{-\infty}I^{\alpha}_x u(x)=\frac{1}{\Gamma(\alpha)}\int^x_{-\infty} (x-\xi)^{\alpha-1}u(\xi)d\xi
\quad \mbox{and}\quad
_{x}I^{\alpha}_{\infty} u(x)=\frac{1}{\Gamma(\alpha)}\int^{\infty}_{x}(\xi-x)^{\alpha-1}u(\xi)d\xi.
$$
Furthermore, for $u\in L^p(\R)$, $p\geq 1$, we have
$$
\mathcal{F}({_{-\infty}}I_{x}^{\alpha}u(x)) = (i\omega)^{-\alpha}\widehat{u}(\omega)\quad \quad \mbox{and}\quad \quad \mathcal{F}({_{x}}I_{\infty}^{\alpha}u(x)) = (-i\omega)^{-\alpha}\widehat{u}(\omega),
$$
and for $u\in C_{0}^{\infty}(\R)$, we have
$$
\mathcal{F}({_{-\infty}}D_{x}^{\alpha}u(x)) = (i\omega)^{\alpha}\widehat{u}(\omega)\quad \quad \mbox{and}\quad \quad \mathcal{F}({_{x}}D_{\infty}^{\alpha}u(x)) = (-i\omega)^{\alpha}\widehat{u}(\omega),
$$

In order to establish the variational structure which enables us to reduce the existence of solutions of (FHS)$_\lambda$ to find critical points of the
corresponding functional, it is necessary to consider some appropriate function spaces. Denote by $L^p(\mathbb{R},\mathbb{R}^n)$ ($1\leq p <\infty$) the Banach spaces of functions on $\mathbb{R}$ with values in $\mathbb{R}^n$ under
the norms
$$
\|u\|_{L^p}=\Bigl(\int_{\mathbb{R}}|u(t)|^p dt\Bigr)^{1/p},
$$
and $L^{\infty}(\mathbb{R},\mathbb{R}^n)$ is the Banach space of essentially bounded functions from $\mathbb{R}$ into $\mathbb{R}^n$ equipped with the norm
$$
\|u\|_{\infty}=\mbox{ess} \sup\left\{|u(t)|: t\in \mathbb{R} \right\}.
$$
Let $-\infty<a<b<+\infty$, $0< \alpha \leq 1$ and $1<p<\infty$. The fractional derivative space $E_{0}^{\alpha ,p}$ is defined by the closure of $C_{0}^{\infty}([a,b], \mathbb{R}^n)$ with respect to the norm
\begin{equation}\label{norm}
\|u\|_{\alpha ,p} = \left(\int_{a}^{b} |u(t)|^pdt + \int_{a}^{b}|{_{a}}D_{t}^{\alpha}u(t)|^pdt  \right)^{1/p}, \;\;\forall\; u\in E_{0}^{\alpha ,p}.
\end{equation}
Furthermore $(E_{0}^{\alpha ,p}, \|.\|_{\alpha ,p})$ is a reflexive and separable Banach space and can be characterized by $$E_{0}^{\alpha , p} = \{u\in L^{p}([a,b], \mathbb{R}^n) |  {_aD}_{t}^{\alpha}u \in L^{p}([a,b], \mathbb{R}^n)\;\mbox{and}\;u(a) = u(b) = 0\}.$$
\begin{proposition}\label{FC-FEprop3}
\cite{FJYZ} Let $0< \alpha \leq 1$ and $1 < p < \infty$. For all $u\in E_{0}^{\alpha ,p}$, we have
\begin{equation}\label{FC-FEeq3}
\|u\|_{L^{p}} \leq \frac{(b-a)^{\alpha}}{\Gamma (\alpha +1)} \|{_aD}_{t}^{\alpha}u\|_{L^{p}}.
\end{equation}
If $\alpha > 1/p$ and $\frac{1}{p} + \frac{1}{q} = 1$, then
\begin{equation}\label{FC-FEeq4}
\|u\|_{\infty} \leq \frac{(b-a)^{\alpha -1/p}}{\Gamma (\alpha)((\alpha - 1)q +1)^{1/q}}\|{_aD}_{t}^{\alpha}u\|_{L^{p}}.
\end{equation}
\end{proposition}

\noindent
By (\ref{FC-FEeq3}), we can consider in $E_{0}^{\alpha ,p}$ the following norm
\begin{equation}\label{FC-FEeq5}
\|u\|_{\alpha ,p} = \|{_aD}_{t}^{\alpha}u\|_{L^{p}},
\end{equation}
which is equivalent to (\ref{norm}).

\begin{proposition}\label{FC-FEprop4}
\cite{FJYZ} Let $0< \alpha \leq 1$ and $1 < p < \infty$. Assume that $\alpha > \frac{1}{p}$ and $\{u_{k}\} \rightharpoonup u$ in $E_{0}^{\alpha ,p}$. Then
$u_{k} \to u$ in $C[a,b]$, i.e.
$$
\|u_{k} - u\|_{\infty} \to 0,\;k\to \infty.
$$
\end{proposition}


For $\alpha>0$, consider the Liouville-Weyl fractional spaces
$$
I^{\alpha}_{-\infty}=\overline{C^{\infty}_0(\mathbb{R},\mathbb{R}^n)}^{\|\cdot\|_{I^{\alpha}_{-\infty}}},
$$
where
\begin{equation}\label{eqn:defn Rnorm}
\|u\|_{I^{\alpha}_{-\infty}}=\Bigl(\int_{\mathbb{R}}u^2(x)dx+ \int_{\mathbb{R}}|_{-\infty}D^{\alpha}_x u(x)|^2dx\Bigr)^{1/2}.
\end{equation}
Furthermore, we introduce the fractional Sobolev space $H^{\alpha}(\mathbb{R},\mathbb{R}^n)$ of order  $0<\alpha<1$ which is defined as
\begin{equation}\label{eqn:alphanorm}
H^{\alpha}=\overline{C^{\infty}_0(\mathbb{R},\mathbb{R}^n)}^{\|\cdot\|_{\alpha}},
\end{equation}
where
$$
\|u\|_{\alpha}=\Bigl(\int_{\mathbb{R}}u^2(x)dx+ \int_{\mathbb{R}}|w|^{2\alpha}\widehat{u}^2(w)dw\Bigr)^{1/2}.
$$
Note that, a function $u\in L^2(\mathbb{R},\mathbb{R}^n)$ belongs to $I^{\alpha}_{-\infty}$ if and only if
$$
|w|^{\alpha}\widehat{u}\in L^2(\mathbb{R},\mathbb{R}^n).
$$
Therefore, $I^{\alpha}_{-\infty}$ and $H^{\alpha}$ are equivalent with equivalent norm, for more details see \cite{ErvinR06}.
\begin{lemma}\label{Lem:LinftyContH}\cite[Theorem 2.1]{Torres12}
If $\alpha>1/2$, then $H^{\alpha}\subset C(\mathbb{R},\mathbb{R}^n)$ and there is a constant $C_\infty=C_{\alpha,\infty}$ such that
\begin{equation}\label{12}
\|u\|_{\infty}=\sup_{x\in \mathbb{R}}|u(x)|\leq C_\infty \|u\|_{\alpha}.
\end{equation}
\end{lemma}

\begin{remark}\label{Rem:Lp}
From Lemma \ref{Lem:LinftyContH}, we know that if $u\in H^{\alpha}$ with $1/2<\alpha<1$, then $u\in L^p(\mathbb{R},\mathbb{R}^n)$ for all $p\in [2,\infty)$, since
$$
\int_{\mathbb{R}}|u(x)|^p dx \leq \|u\|^{p-2}_{\infty}\|u\|^2_{L^2}.
$$
\end{remark}

Now, we introduce the fractional space which we will use to construct the variational framework for (FHS)$_\lambda$. Let
$$
X^{\alpha}=\Bigl\{u\in H^{\alpha}: \int_{\mathbb{R}}[|_{-\infty}D^{\alpha}_{t}u(t)|^2+(L(t)u(t),u(t))]dt<\infty\Bigr\},
$$
then $X^{\alpha}$ is a reflexive and separable Hilbert space with the inner product
$$
\langle u,v \rangle_{X^{\alpha}}=\int_{\mathbb{R}}[(_{-\infty}D^{\alpha}_{t}u(t),_{-\infty}D^{\alpha}_{t}v(t))+(L(t)u(t),v(t))]dt
$$
and the corresponding norm is
$$
\|u\|^2_{X^{\alpha}}=\langle u,u \rangle_{X^{\alpha}}.
$$
For $\lambda>0$, we also need the following inner product
$$
\langle u,v \rangle_{X^{\alpha,\lambda}}=\int_{\mathbb{R}}[(_{-\infty}D^{\alpha}_{t}u(t),_{-\infty}D^{\alpha}_{t}v(t))+\lambda(L(t)u(t),v(t))]dt
$$
and the corresponding norm is
$$
\|u\|^2_{X^{\alpha,\lambda}}=\langle u,u \rangle_{X^{\alpha,\lambda}}.
$$

\begin{lemma}\label{Lem:XcontH}
\cite{ZhangTorres} Suppose $L(t)$ satisfies {\rm ($\mathcal{L}$)$_1$} and {\rm ($\mathcal{L}$)$_2$}, then $X^{\alpha}$ is continuously embedded in $H^{\alpha}$.
\end{lemma}

\begin{remark}\label{keynta}
{\rm
Under the same conditions of Lemma \ref{Lem:XcontH}, for all
$\lambda\geq \frac{1}{c C_\infty^2 \, meas \{l<c\}}$, we also obtain
\begin{equation}\label{13}
\int_{\mathbb{R}}|u(t)|^2 dt\leq \frac{C_\infty^2\, meas\{l<c\}}{1-C_\infty^2\, meas\{l<c\}}\|u\|_{X^{\alpha,\lambda}}=\frac{1}{\Theta}\|u\|_{X^{\alpha,\lambda}}^2
\end{equation}
and
\begin{equation}\label{14}
\|u\|_\alpha^2\leq \Bigl(1+\frac{C_{\infty}^2\, meas\{l<c\}}{1-C_{\infty}^2\, meas \{l<c\}}\Bigr)\|u\|_{X^\alpha}^2=(1+\frac{1}{\Theta})\|u\|^2_{X^{\alpha,\lambda}}.
\end{equation}
Furthermore, for every $p\in (2,\infty)$ and $\lambda\geq \frac{1}{c C_\infty^2 \, meas \{l<c\}}$, we have
\begin{equation}\label{15}
\begin{split}
\int_{\mathbb{R}}|u(t)|^p dt\leq \mathcal{K}_{p}^{p}\|u\|_{X^{\alpha,\lambda}}^p.
\end{split}
\end{equation}
where $\mathcal{K}_{p}^{p} = \frac{1}{\Theta^{\frac{p}{2}}\, (meas\{l<c\})^{\frac{p-2}{2}}}$. For more details, see \cite{Torres15, ZhangTorres}.}
\end{remark}

\section{ Proof of Theorem \ref{Thm:MainTheorem1}}

The aim of this section is to establish the proof of Theorem \ref{Thm:MainTheorem1}. Consider the functional $I: X^{\alpha,\lambda}\rightarrow \mathbb{R}$ given by
\begin{equation}\label{mt01}
\begin{aligned}
I_\lambda (u)&=\int_{\mathbb{R}}\Bigl[\frac{1}{2}|_{-\infty}D_t^{\alpha}u(t)|^2+\frac{1}{2}(\lambda L(t)u(t),u(t))-W(t,u(t))\Bigr]dt\\
&=\dfrac{1}{2}\|u\|^2_{X^{\alpha,\lambda}}-\int_{\mathbb{R}}W(t,u(t))dt.
\end{aligned}
\end{equation}
Under the conditions of Theorem \ref{Thm:MainTheorem1}, we note that $I\in C^1(X^{\alpha,\lambda},\mathbb{R})$, and
\begin{equation}\label{mt02}
I'_\lambda (u)v=\int_{\mathbb{R}}\Bigl[(_{-\infty}D_t^{\alpha}u(t), _{-\infty}D_t^{\alpha}v(t))+(\lambda L(t)u(t),v(t))-(\nabla W(t,u(t)),v(t))\Bigr]dt
\end{equation}
for all $u$, $v\in X^{\alpha}$. In particular we have
\begin{equation}\label{mt03}
I'_\lambda (u)u =\|u\|^2_{X^{\alpha,\lambda}}-\int_{\mathbb{R}}(\nabla W(t,u(t)),u(t))dt.
\end{equation}
\begin{remark}\label{ineq}
It follows from $(W_1)$ and $(W_4)$ that
$$
|\nabla W(t,u)|^\sigma \leq \frac{C_0}{2}|\nabla W(t,u)||u|^{\sigma +1 }\;\;\mbox{for}\;\;|u|\geq R.
$$
Thus, by ($W_1$), for any $\epsilon >0$, there is $C_\epsilon >0$ such that
\begin{equation}\label{mt04}
|\nabla W(t,u)| \leq \epsilon |u| + C_\epsilon |u|^{p-1},\;\;\forall (t,u) \in \mathbb{R} \times \mathbb{R}^N
\end{equation}
and
\begin{equation}\label{mt05}
|W(t,u)| \leq \frac{\epsilon}{2}|u|^2 + \frac{C_\epsilon}{p}|u|^p\;\;\forall (t,u)\in \mathbb{R} \times \mathbb{R}^N,
\end{equation}
where $p = \frac{2\sigma }{\sigma -1}>2$.
\end{remark}

\begin{lemma}\label{GC1}
Suppose that {\rm ($\mathcal{L}$)$_1$}-{\rm ($\mathcal{L}$)$_3$}, {\rm(W$_1$)} and {\rm(W$_2$)} are satisfied. Then
\begin{enumerate}
\item[\fbox{i}] There exists $\rho>0$ and $\eta>0$ such that
$$
\inf_{\|u\|_{X^{\alpha,\lambda}}=\rho} I_\lambda (u)>\eta\quad for \,\, all \,\, \lambda\geq \frac{1}{cC_\infty^2 \, meas \{l<c\}}.
$$
\item[\fbox{ii}] Let $\rho>0$ defined in $(i)$, then there exists $e\in X^{\alpha,\lambda}$ with
$\|e\|_{X^{\alpha,\lambda}}>\rho$ such that $I_\lambda(e)<0$ for all $\lambda>0$.
\end{enumerate}
\end{lemma}
\begin{proof}
\begin{enumerate}
\item[\fbox{i}] By (\ref{mt05}) and Remark \ref{keynta}, we obtain
$$
\begin{aligned}
I_{\lambda}(u) & = \frac{1}{2}\|u\|_{X^{\alpha, \lambda}}^2 - \int_{\mathbb{R}} W(t, u(t))dt\\
& \geq \frac{1}{2}\|u\|_{X^{\alpha, \lambda}}^{2} - \frac{\epsilon}{2}\int_{\mathbb{R}}|u(t)|^2dt - \frac{C_\epsilon}{p}\int_{\mathbb{R}} |u(t)|^pdt\\
&\geq \frac{1}{2}\left( 1 - \frac{\epsilon}{\Theta}\right)\|u\|_{X^{\alpha, \lambda}}^{2} - \frac{C_\epsilon}{p\Theta^{\frac{p}{2}}(meas\{l<c\})^{\frac{p-2}{2}}}\|u\|_{X^{\alpha, \lambda}}^{p}.
\end{aligned}
$$
Let $\epsilon >0$ small enough such that $1- \frac{\epsilon}{\Theta}>0$ and $\|u\|_{X^{\alpha, \lambda}} = \rho$. Since $p>2$, taking $\rho$ small enough such that
$$
\frac{1}{2}\left( 1-\frac{\epsilon}{\Theta} \right) - \frac{C_\epsilon}{p\Theta^{\frac{p}{2}}(meas\{l<c\})^{\frac{p-2}{2}}} \rho^{p-2}>0.
$$
Therefore
$$
I_\lambda (u) \geq \rho^2\left[ \frac{1}{2}\left( 1-\frac{\epsilon}{\Theta} \right) - \frac{C_\epsilon}{p \Theta^{\frac{p}{2}}(meas\{l<c\})^{\frac{p-2}{2}}} \rho ^{p-2}\right] : = \eta >0.
$$

\item[\fbox{ii}] By $(\mathcal{L})_3$ and without loss of generality let $T = (-\varrho, \varrho )\subset J$ such that $L(t) \equiv 0$. Let $\psi \in C_{0}^{\infty}(\mathbb{R}, \R^n)$ such that $supp(\psi) \subset (-\tau, \tau)$, for some $\tau< \varrho$. Hence
\begin{equation}\label{ss00}
\begin{aligned}
0&\leq \int_{\R}\langle L(t)\psi , \psi\rangle dt = \int_{supp(\psi)} \langle L(t)\psi, \psi\rangle dt \leq \int_{-\tau}^{\tau} \langle L(t)\psi, \psi\rangle dt \leq \int_{T} \langle L(t)\psi, \psi\rangle dt = 0.
\end{aligned}
\end{equation}
On the other hand, by $(W_3)$, for any $\epsilon >0$, there exists $R>0$ such that
$$
W(t,u) > \frac{|u|^2}{\epsilon} - \frac{R^2}{\epsilon}\;\;\mbox{for all}\;\;|u|\geq R.
$$
Then, by taking $\epsilon \to 0$ we get
\begin{equation}\label{ss01}
\lim_{|\sigma| \to \infty} \int_{supp(\psi)} \frac{W(t, \sigma \psi)}{|\sigma|^2}dt = +\infty.
\end{equation}
Hence, by (\ref{ss00}) and (\ref{ss01}) we obtain
\begin{equation}\label{ss02}
\frac{I_\lambda (\sigma \psi)}{|\sigma|^2}  = \frac{1}{2}\int_{\R} |_{-\infty}D_{t}^{\alpha}\psi(t)|^2dt - \int_{\R} \frac{W(t, \sigma \psi)}{|\sigma|^2}dt \to -\infty,
\end{equation}
as $|\sigma| \to \infty$. Therefore, if $\sigma_0$ is large enough and $e = \sigma_0 \psi$ one gets $I_\lambda (e) <0$.
\end{enumerate}
\end{proof}

Since we have loss of compactness we need the following compactness results to recover the Cerami condition for $I_\lambda$.

\begin{lemma}\label{PS1}
Suppose that $(\mathcal{L})_1 - (\mathcal{L})_3$, $(W_1) - (W_4)$ be satisfied. If $u_n \rightharpoonup u$ in $X^{\alpha, \lambda}$, then
\begin{equation}\label{mt07}
I_\lambda (u_n - u) = I_\lambda (u_n) - I_\lambda (u) + o(1)\;\;\mbox{as}\;\;n\to +\infty
\end{equation}
and
\begin{equation}\label{mt08}
I'_\lambda(u_n-u) = I'_\lambda(u_n) - I'_\lambda (u) + o(1) \;\;\mbox{as}\;\;n\to +\infty.
\end{equation}
In particular, if $I_\lambda (u_n) \to c$ and $I'_\lambda(u_n) \to 0$, then $I'_\lambda(u) = 0$ after passing to a subsequence.
\end{lemma}

\begin{proof}
Since $u_n \rightharpoonup u$ in $X^{\alpha, \lambda}$, we have $\langle u_n -u, u\rangle_{X^{\alpha, \lambda}} \to 0$ as $n\to \infty$, which implies that
$$
\begin{aligned}
\|u_n\|_{X^{\alpha, \lambda}}^{2}  = \|u_n - u\|_{X^{\alpha, \lambda}}^{2} + \|u\|_{X^{\alpha, \lambda}}^{2} + o(1).
\end{aligned}
$$
Therefore, to obtain (\ref{mt07}) and (\ref{mt08}) it suffices to check that
\begin{equation}\label{mt09}
\int_{\mathbb{R}} [W(t,u_n) - W(t, u_n-u) - W(t,u)]dt = o(1)
\end{equation}
and
\begin{equation}\label{mt10}
\sup_{\varphi \in X^{\alpha, \lambda}, \|\varphi\|_{\alpha, \lambda} =1} \int_{\mathbb{R}} \langle \nabla W(t, u_n) - \nabla W(t, u_n-u)- \nabla W(t,u), \varphi\rangle dt = o(1).
\end{equation}
Here, we only prove (\ref{mt10}), the verification of (\ref{mt09}) is similar. In fact, let
\begin{equation}\label{mt11}
\mathcal{A}:= \lim_{n\to \infty} \sup_{\varphi \in X^{\alpha, \lambda}, \|\varphi\|_{\alpha, \lambda} =1} \int_{\mathbb{R}} \langle \nabla W(t, u_n) - \nabla W(t, u_n-u)- \nabla W(t,u), \varphi\rangle dt.
\end{equation}
If $\mathcal{A}>0$, then, there exists $\varphi_0 \in X^{\alpha, \lambda}$ with $\|\varphi_0\|_{X^{\alpha, \lambda}} = 1$ such that
$$
\left| \int_{\mathbb{R}} \langle \nabla W(t,u_n) - \nabla W(t, u_n-u) - \nabla W(t,u), \varphi_0\rangle dt \right| \geq \frac{\mathcal{A}}{2}
$$
for $n$ large enough. Now, from (\ref{mt04}) and Young's inequality, there exist $C_1$, $C_2$  and $C_3 >0$ such that
$$
\begin{aligned}
&|\langle \nabla W(t, u_n) - \nabla W(t, u_n-u), \varphi_0\rangle|
\leq C_1 \left( \epsilon |u|^2 + \epsilon |u_n-u|^2 + \epsilon |\varphi_0|^2 + C_2|u|^p + \epsilon |u_n - u|^{p} + C_3|\varphi_0|^p\right).
\end{aligned}
$$
Hence, there exists $C_4, C_5, C_6 >0$ such that
$$
\begin{aligned}
|\langle \nabla W(t, u_n) - \nabla W(t, u_n - u) &- \nabla W(t,u), \varphi_0\rangle|\\
& \leq C_4 \left( \epsilon |u|^2 + \epsilon |u_n - u|^2 + \epsilon |\varphi_0|^2 + C_5|u|^p + \epsilon |u_n - u|^p + C_6|\varphi_0|^p\right).
\end{aligned}
$$
Let
$$
h_n(t) = \max\{|\langle \nabla W(t, u_n) - \nabla W(t, u_n - u) - \nabla W(t,u), \varphi_0 \rangle| - C_4 \epsilon (|u_n - u|^2 + |u_n - u|^p), 0\}.
$$
So
$$
0\leq h_n(t) \leq C_4 (\epsilon |u|^2 + \epsilon |\varphi_0|^2 + C_5|u|^p + C_6|\varphi_0|^p).
$$
By the Lebesgue dominated convergence Theorem and the fact $u_n \to u$ a.e. in $\mathbb{R}$, , we can get
$$
\int_{\mathbb{R}} h_n(t)dt \to 0\;\;\mbox{as}\;\;n \to \infty.
$$
From where
$$
\int_{\mathbb{R}} |\langle \nabla W(t, u_n(t)) - \nabla W(t, u_n(t) - u(t)) - \nabla W(t,u(t)), \varphi_0(t) \rangle|dt \to 0\;\;\mbox{as}\;\; n\to \infty,
$$
which is a contradiction. Hence $\mathcal{A} = 0$.

Furthermore, if $I_\lambda (u_n) \to c$ and $I'_\lambda (u_n) \to 0$ as $n \to \infty$, by (\ref{mt07}) and (\ref{mt08}), we get
$$
I_\lambda (u_n - u) \to c-I_\lambda (u) + o(1)
$$
and
$$
I'_\lambda (u_n - u) = -I'_\lambda(u)\;\;\mbox{as}\;\;n \to +\infty.
$$
Now, for every $\varphi \in C_{0}^{\infty}(\mathbb{R}, \mathbb{R}^n)$ we have
$$
I'_\lambda(u)\varphi = \lim_{n \to \infty} I'_\lambda (u_n) \varphi = 0.
$$
Consequently, $I'_\lambda (u) = 0$.
\end{proof}

\begin{lemma}\label{cerami1}
Suppose that $(\mathcal{L})_1 - (\mathcal{L})_3$, $(W_1) - (W_4)$ be satisfied and let $c\in \mathbb{R}$. Then each $(Ce)_c$-sequence of $I_\lambda$ is bounded in $X^{\alpha, \lambda}$.
\end{lemma}

\begin{proof}
Suppose that $\{u_n\} \subset X^{\alpha, \lambda}$ is a $(Ce)_c$ sequence for $c>0$, namely
\begin{equation}\label{mt12}
I_\lambda (u_n) \to c,\quad (1+\|u_n\|_{X^{\alpha, \lambda}})I'_\lambda (u_n) \to 0\;\;\mbox{as}\;\;n\to \infty.
\end{equation}
Therefore
\begin{equation}\label{mt13}
c- o_n(1) = I_\lambda (u_n) - \frac{1}{2}I'_\lambda(u_n)u_n = \int_{\mathbb{R}} H(t, u_n(t))dt.
\end{equation}
By contradiction, suppose that there is a subsequence, again denoted by $\{u_n\}$, such that $\|u_n\|_{X^{\alpha, \lambda}} \to +\infty$ as $n\to +\infty$. Taking $v_n = \frac{u_n}{\|u_n\|_{X^{\alpha, \lambda}}}$, we get that $\{v_n\}$ is bounded in $X^{\alpha, \lambda}$ and $\|v_n\|_{X^{\alpha, \lambda}} = 1$. Moreover, we have
$$
o(1) = \frac{\langle I'_\lambda (u_n), u_n\rangle}{\|u_n\|_{X^{\alpha, \lambda}}^{2}} = 1- \int_{\mathbb{R}} \frac{\langle \nabla W(t,u_n), u_n \rangle}{\|u_n\|_{X^{\alpha, \lambda}}^{2}},
$$
as $n \to \infty$, which implies 
\begin{equation}\label{lim}
\int_{\mathbb{R}} \frac{\langle \nabla W(t, u_n), v_n\rangle}{|u_n|} |v_n|dt = \int_{\mathbb{R}} \frac{\langle \nabla W(t, u_n), u_n\rangle}{\|u_n\|_{X^{\alpha, \lambda}}^{2}} \to 1.
\end{equation}
For $r\geq 0$, let
$$
h( r ) := \inf\{H(t, u):\;\;t\in \mathbb{R}, \;\;|u|\geq r\}.
$$
From ($W_2$) we have $h ( r) >0$ for all $r>0$. Furthermore, by ($W_2$) and $(W_4)$, for $|u|\geq r$,
\begin{equation}\label{mt14}
C_0 H(t,u) \geq \frac{|\nabla W(t,u)|^\sigma}{|u|^\sigma} = \left( \frac{|\nabla W(t,u)||u|}{|u|^2} \right)^{\sigma} \geq \left( \frac{\langle \nabla W(t,u), u\rangle}{|u|^2}\right)^{\sigma} \geq \left( \frac{2W(t,u)}{|u|^2} \right)^{\sigma},
\end{equation}
it follows from $(W_3)$ and the definition of $h(r)$ that
$$
h( r) \to \infty\;\;\mbox{as}\;\;r\to \infty.
$$
For $0\leq a < b$, let
$$
\Omega_{n}(a, b) :=\{t\in \mathbb{R}:\;\;a \leq |u_n(t)|< b\}
$$
and
$$
C_{a}^{b} : = \inf \left\{\frac{H(t, u)}{|u|^2}: \;\;t\in \mathbb{R}\;\;\mbox{and}\;\;u\in \mathbb{R}^N\;\;\mbox{with}\;\;a \leq |u|< b \right\}.
$$
By ($W_1$), for any $\epsilon >0$, there is $\delta >0$ such that
$$
|\nabla W(t,u)| \leq \frac{\epsilon}{\mathcal{K}_{2}^{2}}|u|\;\;\mbox{for all}\;\;|t|\leq \delta.
$$
Consequently
\begin{equation}\label{mt15}
\begin{aligned}
\int_{\Omega_{n}(0, \delta)} \frac{|\nabla W(t, u_n)|}{|u_n|}|v_n|^2dt \leq \int_{\Omega_n(0, \delta)} \frac{\epsilon}{\mathcal{K}_{2}^{2}}|v_n|^2dt\leq \frac{\epsilon}{\mathcal{K}_{2}^{2}}\|v_n\|_{L^2}^{2} \leq \epsilon,\;\;\forall n.
\end{aligned}
\end{equation}
We note that
$$
H(t,u_n)\geq C_{a}^{b}|u_n|^2\quad \mbox{for all} \quad t \in \Omega_n (a,b),
$$
consequently, by (\ref{mt13}) we get
\begin{equation}\label{mt16}
\begin{aligned}
c-o_n(1) &= \int_{\Omega_n(0,a)} H(t,u_n)dt + \int_{\Omega_n(a,b)}H(t,u_n)dt + \int_{\Omega_n(b, +\infty)}H(t,u_n)dt\\
&\geq \int_{\Omega_n (0,a)} H(t,u_n)dt + C_{a}^{b}\int_{\Omega_n(a,b)}|u_n|^2dt + \int_{\Omega_n(b, +\infty)}H(t,u_n)dt\\
&= \int_{\Omega_n (0,a)} H(t,u_n)dt + C_{a}^{b}\int_{\Omega_n(a,b)}|u_n|^2dt + h(b)meas(\Omega_n(b,+\infty)).
\end{aligned}
\end{equation}
Since $h( r) \to +\infty$ as $r\to +\infty$, for $p< q < \infty$ it follows from (\ref{mt16}) that
\begin{equation}\label{mt17}
\begin{aligned}
\int_{\Omega_{n}(b, +\infty)}|v_n|^pdt &\leq \left( \int_{\Omega_n(b, +\infty)}|v_n|^qdt \right)^{\frac{p}{q}} meas(\Omega_n(b+\infty))^{\frac{q-p}{q}}\\
&\leq \|v_n\|_{L^q}^{p} \left( \frac{c-o_n(1)}{h(b)} \right)^{\frac{q-p}{p}}\leq \mathcal{K}_{q}^{p}\left( \frac{c-o_n(1)}{h(b)} \right)^{\frac{q-p}{p}} \to 0
\end{aligned}
\end{equation}
as $b\to +\infty$, where $p=\frac{2\sigma}{\sigma -1}>2$. Furthermore, by ($W_4$) and the H\"older inequality, we can choose $R>0$ large enough such that
\begin{equation}\label{mt18}
\begin{aligned}
\left|\int_{\Omega_n(R, +\infty)} \frac{\langle \nabla W(t, u_n), u_n\rangle}{\|u_n\|_{X^{\alpha, \lambda}}^{2}} dt \right|& \leq \int_{\Omega_n(R,+\infty)} \frac{|\nabla W(t,u_n)|}{|u_n|}|v_n|^2dt\\
&\leq \left( \int_{\Omega_n(R, +\infty)} \frac{|\nabla W(t,u_n)|^\sigma}{|u_n|^\sigma} \right)^{1/\sigma} \left( \int_{\Omega_n(R,+\infty)} |v_n|^pdt \right)^{\frac{\sigma-1}{\sigma}}\\
&\leq \left(\int_{\Omega_n(R, +\infty)}C_0H(t,u_n)dt  \right)^{1/\sigma} \left( \int_{\Omega_n(R, +\infty)} |v_n|^pdt \right)^{\frac{\sigma -1}{\sigma}}\\
&\leq C_0^{1/\sigma} (c-o_n(1))^{1/\sigma} \left( \int_{\Omega_n(R, +\infty)}|v_n|^pdt \right)^{\frac{\sigma -1}{\sigma}}\\&< \epsilon.
\end{aligned}
\end{equation}
Now, by using (\ref{mt16}) again, we get
$$
\int_{\Omega_n(\delta, R)} |v_n|^2dt = \frac{1}{\|u_n\|_{X^{\alpha, \lambda}}^{2}} \int_{\Omega_n(\delta, R)}|u_n|^2dt \leq \frac{c-o_n(1)}{C_{\delta}^{R}\|u_n\|_{X^{\alpha, \lambda}}^{2}} \to 0
$$
as $n\to \infty$. Then, for $n$ large enough, by the continuity of $\nabla W$ one has
\begin{equation}\label{mt19}
\int_{\Omega_n(\delta, R)} \frac{|\nabla W(t, u_n)|}{|u_n|}|v_n|^2dt \leq K \int_{\Omega_n(\delta, R)} |v_n|^2dt < \epsilon.
\end{equation}
Hence, by (\ref{mt15}), (\ref{mt18}) and (\ref{mt19}) we have
$$
\int_{\mathbb{R}} \frac{\langle \nabla W(t,u_n), v_n \rangle}{|u_n|}|v_n|dt \leq \int_{\mathbb{R}} \frac{|\nabla W(t, u_n)|}{|u_n|}|v_n|^2dt \leq 3 \epsilon<1,
$$
for $n$ large enough, a contradiction with (\ref{lim}) and then $\{u_n\}$ is bounded in $X^{\alpha, \lambda}$.
\end{proof}

\begin{lemma}\label{cerami2}
Suppose that $(\mathcal{L})_1-(\mathcal{L})_3$, $(W_1) - (W_4)$ be satisfied. Then, for any $\mathfrak{C}>0$, there exists $\Lambda_1 = \Lambda(\mathfrak{C})>0$ such that $I_\lambda$ satisfies $(Ce)_c$ condition for all $c \leq \mathfrak{C}$ and $\lambda > \Lambda_1$.
\end{lemma}
\begin{proof}
For any $\mathfrak{C}>0$, suppose that $\{u_n\} \subset X^{\alpha, \lambda}$ is a $(Ce)_c$ sequence for $c\leq \mathfrak{C}$, namely
$$
I_\lambda (u_n) \to c,\quad (1+\|u_n\|_{X^{\alpha, \lambda}})I'_\lambda (u_n) \to 0\;\;\mbox{as}\;\;n\to \infty.
$$
By Lemma \ref{cerami1}, $\{u_n\}$ is bounded. Therefore, there exists $u\in X^{\alpha, \lambda}$ such that $u_n \rightharpoonup u $ in $X^{\alpha, \lambda}$ and $u_n \to u$ a.e. in $\mathbb{R}$.

Let $w_n:= u_n-u$. 
By Lemma \ref{PS1}  we get
$$
I'_\lambda (u) = 0,\quad I_\lambda (w_n) \to c-I_\lambda (u) \quad \mbox{and}\quad I'_\lambda (w_n) \to 0\;\;\mbox{as}\;\;n\to \infty.
$$
Next
\begin{equation}\label{mt21}
I_\lambda (u) = I_\lambda (u) - \frac{1}{2}I'_\lambda (u)u = \int_{\mathbb{R}} H(t,u)dt \geq 0,
\end{equation}
and
\begin{equation}\label{mt22}
\int_{\mathbb{R}}H(t,w_n)dt  \to c- I_\lambda (u).
\end{equation}
Therefore, for $c \leq \mathfrak{C}$, we get
\begin{equation}\label{mt23}
\int_{\mathbb{R}} H(t, w_n)dt \leq \mathfrak{C} + o_n(1).
\end{equation}
On the other hand, by $(\mathcal{L})_1$ and since $w_n \to 0$ in $L^2_{loc}(\mathbb{R}, \mathbb{R}^N)$, we have
\begin{equation}\label{mt24}
\begin{aligned}
\|w_n\|_{L^2}^{2} 
\leq \frac{1}{\lambda c}\int_{\{l\geq c\}}\lambda \langle L(t)w_n,w_n \rangle dt + o_n(1)\leq \frac{1}{\lambda c} \|w_n\|_{X^{\alpha, \lambda}}^{2} + o_n(1).
\end{aligned}
\end{equation}
Let $p < q < \infty$, where $p = \frac{2\sigma}{\sigma -1}$. Using Remark \ref{keynta} and H\"older inequality we obtain
\begin{equation}\label{mt25}
\begin{aligned}
\int_{\mathbb{R}}|w_n|^pdt 
&= \int_{\mathbb{R}} |w_n|^{\frac{2(q-p)}{q-2}}|w_n|^{\frac{q(p-2)}{q-2}}dt \leq \|w_n\|_{L^2}^{\frac{2(q-p)}{q-2}}\|w_n\|_{L^q}^{\frac{q(p-2)}{q-2}}\\
&\leq \mathcal{K}_{q}^{\frac{q(p-2)}{q-2}} \left( \frac{1}{\lambda c} \right)^{\frac{q-p}{q-2}}\|w_n\|_{X^{\alpha, \lambda}}^{p} + o_n(1).
\end{aligned}
\end{equation}
Furthermore, for $|u| \leq R$ (where $R$ is defined in (W$_4$)), from (\ref{mt04}), we get
$$
|\nabla W(t, u)| \leq (\epsilon + C_\epsilon R^{p-2})|u| = \tilde{C}|u|.
$$
It follows from (\ref{mt24}) that
$$
\begin{aligned}
\int_{\{t\in \mathbb{R}:\;\;|w_n(t)|\leq R\}} \langle \nabla W(t,w_n), w_n\rangle dt & \leq \int_{\{t\in \mathbb{R}:\;\;|w_n(t)|\leq R\}}|\nabla W(t, w_n)||w_n|dt  \\
&\leq \tilde{C}\int_{\{t\in \mathbb{R}:\;\;|w_n(t)|\leq R\}}|w_n|^2dt\leq \frac{\tilde{C}}{\lambda c}\|w_n\|_{X^{\alpha, \lambda}}^{2} + o_n(1).
\end{aligned}
$$
On the other hand, from (\ref{mt25}) and the H\"older inequality we obtain
$$
\begin{aligned}
\int_{\{t\in \mathbb{R}:\;\;|w_n(t)|> R\}} \langle \nabla W(t,w_n), w_n\rangle dt &\leq \int_{\{t\in \mathbb{R}:\;\;|w_n(t)|> R\}}|\nabla W(t,w_n)||w_n|dt \\
&\leq \int_{\{t\in \mathbb{R}:\;\;|w_n(t)|> R\}} \frac{|\nabla W(t, w_n)|}{|w_n|}|w_n|^2dt\\
&\leq \left( \int_{\{t\in \mathbb{R}:\;\;|w_n(t)|> R\}} \frac{|\nabla W(t, w_n)|^\sigma}{|w_n|^\sigma}dt\right)^{1/\sigma} \left( \int_{\{t\in \mathbb{R}:\;\;|w_n(t)|> A\}} |w_n|^p \right)^{\frac{2}{p}}\\
&\leq \left( C_0\int_{\mathbb{R}} H(t, w_n)dt \right)^{1/\sigma} \|w_n\|_{p}^{2}\\
&\leq (C_0 \mathfrak{C})^{1/\sigma} \mathcal{K}_{q}^{\frac{2q(p-2)}{p(q-2)}}\left( \frac{1}{\lambda c} \right)^{\frac{2(q-p)}{p(q-2)}} \|w_n\|_{X^{\alpha, \lambda}}^{2} + o_n(1).
\end{aligned}
$$
Therefore
$$
\begin{aligned}
o_n(1) & = \langle I'_\lambda (w_n), w_n\rangle = \|w_n\|_{X^{\alpha, \lambda}}^{2} - \int_{\mathbb{R}}\langle \nabla W(t,w_n), w_n\rangle dt\\
&= \|w_n\|_{X^{\alpha, \lambda}}^{2} - \int_{\{t\in \mathbb{R}:\;\;|w_n(t)|\leq R\}} \langle \nabla W(t,w_n), w_n\rangle dt - \int_{\{t\in \mathbb{R}:\;\;|w_n(t)|>R\}} \langle \nabla W(t,w_n), w_n\rangle dt\\
&\geq \left( 1 - \frac{\tilde{C}}{\lambda c}- C^{*} \left( \frac{1}{\lambda c} \right)^{\frac{2(q-p)}{p(q-2)}} \right) \|w_n\|_{X^{\alpha, \lambda}}^{2}+ o_n(1),
\end{aligned}
$$
where $C^* = (C_0 \mathfrak{C})^{1/\sigma} \mathcal{K}_{q}^{\frac{2q(q-p)}{p(q-2)}}$. Now, we choose $\Lambda_1 = \Lambda (\mathfrak{C}) >0$ large enough such that
 $$
 1 - \frac{\tilde{C}}{\lambda c}- C^{*} \left( \frac{1}{\lambda c} \right)^{\frac{2(q-p)}{p(q-2)}} >0\quad \mbox{for all} \;\;\lambda > \Lambda_1.
 $$
Then $w_n \to 0$ in $X^{\alpha, \lambda}$ for all $\lambda > \Lambda_1$.
\end{proof}

\noindent
{\bf Proof of Theorem \ref{Thm:MainTheorem1}}
By Lemmas \ref{GC1}, $I_\lambda$ has the mountain pass geometry and by Lemma \ref{cerami2}, $I_\lambda$ satisfies the $(Ce)_c$-condition. Therefore,
by using mountain pass lemma with Cerami condition \cite{Ekeland}, for any $c_\lambda>0$ defined as follows
$$
c_\lambda=\inf_{g\in \Gamma}\max_{s\in [0,1]}I_\lambda (g(s)),
$$
where
$$
\Gamma=\{g\in C([0,1],X^{\alpha,\lambda})\,|\, g(0)=0, g(1)=e\},
$$
($e$ is defined in Lemma \ref{GC1}-ii), there exists $u_\lambda\in X^{\alpha,\lambda}$ such that
\begin{equation}\label{eqn:ulambda}
I_\lambda(u_\lambda)=c_\lambda\quad\mbox{and}\quad I'_\lambda(u_\lambda)=0.
\end{equation}
That is, (FHS)$_\lambda$ has at least one nontrivial solution for $\lambda>\Lambda(c_\lambda)$ (defined in Lemma \ref{cerami2}).

\section{Concentration phenomena}

In this section, we study the concentration of solutions for problem $(\mbox{FHS})_{\lambda}$ as $\lambda \to \infty$. That is, we focus our attention on the proof
of Theorem \ref{Thm:MainTheorem2}.
\begin{remark}\label{ccnota}
The main difficulty to proof Theorem \ref{Thm:MainTheorem2}, is to show that $c_\lambda$ is bounded form above independent of $\lambda$. Thank to the proof of Lemma \ref{GC1}-ii, we can get a finite upper bound to $c_\lambda$, that is, choose $\psi$ as in the proof of Lemma \ref{GC1}-ii, then by definition of $c_\lambda$, we have
$$
\begin{aligned}
c_\lambda &\leq \max_{\sigma \geq 0} I_\lambda (\sigma \psi)\\
&= \max_{\sigma \geq 0}\left( \frac{\sigma^2}{2}\int_{\R}|{_{-\infty}}D_{t}^{\alpha}\psi(t)|^2 - \int_{\R}W(t, \sigma \psi)dt \right)\\
&= \tilde{c},
\end{aligned}
$$
where $\tilde{c} <+\infty$ is independent of $\lambda$.

As a consequence of the above estimates, we have that $\Lambda(c_\lambda)$ is bounded form below. That is, there exists
$\Lambda_*>0$ such that the conclusion of Theorem \ref{Thm:MainTheorem1} is satisfied for $\lambda>\Lambda_*$.
\end{remark}

Consider $T = [-\varrho, \varrho]$ and the following fractional boundary value problem
\begin{equation}\label{eqn:BVP}
\left\{
  \begin{array}{ll}
   {_{t}}D_{\varrho}^{\alpha} {_{-\varrho}}D_{t}^{\alpha}u  = \nabla W(t, u),\quad t\in (-\varrho, \varrho),\\[0.1cm]
    u(-\varrho) = u(\varrho) = 0.
  \end{array}
\right.
\end{equation}
Associated to (\ref{eqn:BVP}) we have the functional $I: E_{0}^{\alpha} \to \mathbb{R}$ given by
$$
I(u):= \frac{1}{2}\int_{-\varrho}^{\varrho} |{_{-\varrho}}D_{t}^{\alpha}u(t)|^2dt - \int_{-\varrho}^{\varrho}W(t,u(t))dt
$$
and we have that $I\in C^1(E_{0}^{\alpha}, \mathbb{R})$ with
$$
I'(u)v = \int_{-\varrho}^{\varrho} \langle {_{-\varrho}}D_{t}^{\alpha}u(t), {_{-\varrho}}D_{t}^{\alpha}v(t)\rangle dt - \int_{-\varrho}^{\varrho}\langle\nabla W(t,u(t)), v(t) \rangle dt.
$$
Following the ideas of the proof of Theorem \ref{Thm:MainTheorem1}, we can get the following existence result
\begin{theorem}\label{Gthm1}
Suppose that $W$ satisfies $(W_1)-(W_4)$ with $t\in [-\varrho,\varrho]$, then (\ref{eqn:BVP}) has at least one weak nontrivial solution.
\end{theorem}

\noindent
{\bf Proof of Theorem \ref{Thm:MainTheorem2}} We follow the argument in \cite{ZhangTorres}.
For any sequence $\lambda_k \to \infty$, let $u_k = u_{\lambda_k}$ be the critical point of $I_{\lambda_k}$, namely
$$
c_{\lambda_k} = I_{\lambda_k}(u_k)\quad \mbox{and}\quad I'_{\lambda_k}(u_k)=0,
$$
and, by (\ref{mt05}), we get
$$
\begin{aligned}
c_{\lambda_k} &= I_{\lambda_k}(u_k)=\frac{1}{2}\|u_k\|_{X^{\alpha,\lambda}}^2-\int_{\R} W(t,u_k(t))dt \\
&\geq \frac{1}{2}\|u_k\|_{X^{\alpha,\lambda}}^2-\frac{\epsilon}{2}\int_{\R}|u_k|^2dt - \frac{C_\epsilon}{p}\int_{\R}|u_k|^pdt,
\end{aligned}
$$
which implies that $\{u_k\}$ is bounded, due to Remarks \ref{Rem:Lp} and \ref{keynta}.
Therefore, we may assume that $u_k \rightharpoonup \tilde{u}$ weakly in $X^{\alpha,\lambda_k}$.
Moreover, by Fatou's lemma, we have
$$
\begin{aligned}
\int_{\mathbb{R}} l(t) |\tilde{u}(t)|^2dt \leq \liminf_{k\to \infty} \int_{\mathbb{R}} l(t)|u_k(t)|^2dt\leq \liminf_{k\to \infty} \int_{\mathbb{R}} (L(t)u_k(t), u_{k}(t)) dt \leq \liminf_{k\to \infty} \frac{\|u_k\|_{X^{\alpha, \lambda_k}}^2}{\lambda_{k}} = 0.
\end{aligned}
$$
Thus, $\tilde{u} = 0$ a.e. in $\mathbb{R} \setminus J$. Now, for any $\varphi \in C_{0}^{\infty}(T, \mathbb{R}^n)$, since $I'_{\lambda_k}(u_k)\varphi=0$, it is
easy to see that
$$
\int_{-\varrho}^{\varrho} ({_{-\varrho}}D_{t}^{\alpha} \tilde{u}(t), {_{-\varrho}}D_{t}^{\alpha}\varphi(t)) dt - \int_{-\varrho}^{\varrho} (\nabla W(t,\tilde{u}(t)), \varphi (t)) dt=0,
$$
that is, $\tilde{u}$ is a solution of (\ref{eqn:BVP}) by the density of $C_{0}^{\infty}(T, \mathbb{R}^n)$ in $E^{\alpha}$.

Now we show that $u_k \to \tilde{u}$ in $X^{\alpha}$. Since $I'_{\lambda_k}(u_k) u_k=I'_{\lambda_k}(u_k)\tilde{u}=0$, we have
\begin{equation}\label{c6}
\|u_k\|_{X^{\alpha, \lambda_k}}^{2} = \int_{\mathbb{R}} (\nabla W(t, u_k(t)), u_k(t)) dt
\end{equation}
and
\begin{equation}\label{c7}
\langle u_k, \tilde{u} \rangle_{\lambda_k} = \int_{\mathbb{R}} (\nabla W(t, u_k(t)), \tilde{u}(t)) dt,
\end{equation}
which implies that
$$
\lim_{k\to \infty} \|u_{k}\|_{X^{\alpha, \lambda_k}}^{2} = \lim_{k\to \infty} \langle u_k, \tilde{u} \rangle_{X^{\alpha, \lambda_k}} = \lim_{k\to \infty} \langle u_k, \tilde{u} \rangle_{X^{\alpha}} = \|\tilde{u}\|_{X^\alpha}^2.
$$
Furthermore, by the weak semi-continuity of norms we obtain
$$
\|\tilde{u}\|_{X^{\alpha}}^{2} \leq \liminf_{k \to \infty} \|u_k\|_{X^{\alpha}}^{2} \leq \limsup_{k\to \infty}\|u_k\|_{X^\alpha}^{2} \leq \lim_{k\to \infty}\|u_k\|_{X^{\alpha,\lambda_k}}^{2}.
$$
So $u_k \to \tilde{u}$ in $X^{\alpha}$, and $u_k \to \tilde{u}$ in $H^{\alpha}(\mathbb{R}, \mathbb{R}^n)$ as $k\to \infty$.
\qed

\end{document}